\documentclass[12pt]{amsart}
\usepackage{amscd,amssymb,hyperref,amsfonts,xcolor,graphicx, epsfig}
\usepackage[all]{xy}
\usepackage{cite}
\let\OLDthebibliography\thebibliography
\renewcommand\thebibliography[1]{
  \OLDthebibliography{#1}
  \setlength{\parskip}{1pt}
  \setlength{\itemsep}{1pt plus 0.3ex}
}

\newtheorem{theorem}{Theorem}

\newtheorem{prop}{Proposition}

\theoremstyle{definition}

\newtheorem{example}{Example}
\theoremstyle{remark}
\newtheorem{remark}{Remark}
\newtheorem{question}{Problem}

\begin{document}
\title{Multi-switches and virtual knot invariants}

\author{Valeriy Bardakov,~Timur Nasybullov}
\date{}
\begin{abstract} In the paper we introduce a general approach how for a given virtual biquandle multi-switch $(S,V)$ on an algebraic system $X$ (from some category) and a given virtual link  $L$ construct an algebraic system $X_{S,V}(L)$ (from the same category) which is an invariant of $L$. As a corollary we introduce a new quandle invariant for virtual links which generalizes previously known quandle invariants for virtual links.

~
\end{abstract}
\maketitle
\section{Introduction}
\textit{A set-theoretical solution of the Yang-Baxter equation} is a pair $(X, S)$, where $X$ is a set and $S:X\times X\to X\times X$ is a bijective map such that
$$(S\times id)(id \times S)(S\times id)=(id\times S)(S \times id)(id\times S).$$
The problem of studying set-theoretical solutions of the Yang-Baxter equation was formulated by Drinfel'd in \cite{Dri}. If $(X, S)$ is a set-theoretical solution of the Yang-Baxter equation, then the map $S$ is called \textit{a switch} on $X$ (see \cite{FMK}). A pair of switches $(S,V)$ on $X$ is called \textit{a virtual switch} on $X$ if $V^2=id$ and the equality
$$(V\times id)(id \times S)(V\times id)=(id\times V)(S \times id)(id\times V)$$
holds.

Switches and virtual switches are connected with virtual braid groups and virtual links. Using a (virtual) switch one can construct an integer-valued invariant for (virtual) links, so called coloring-invariant \cite[Section~6]{FMK}. Also using a (virtual) switch on a set $X$ it is possible to construct a representation of the (virtual) braid group on $n$ strands by permutations of $X^n$ \cite[Section~2]{FMK}. Moreover, if $X$ is an algebraic system, then under additional conditions using a (virtual) switch on $X$ it is possible to construct a representation of the (virtual) braid group by automorphisms of $X$. The Artin representation $\varphi_A:B_n\to {\rm Aut}(F_n)$
 (see \cite[Section~1.4]{Bir}), the Burau representation $\varphi_B:B_n\to {\rm GL}_n(\mathbb{Z}[t,t^{-1}])$ (see  \cite[Section~3]{Bir}) and their extensions to the virtual braid groups $\widetilde{\varphi}_A:VB_n\to{\rm Aut}(F_n)$, $\widetilde{\varphi}_B:VB_n\to {\rm GL}_n(\mathbb{Z}[t,t^{-1}])$ (see \cite{Kau, Ver}) can be obtained on this way.

Despite the fact that virtual switches can be used for constructing representations of  virtual braid groups, there are representations $VB_n\to{\rm Aut}(G)$, where $G$ is some group, which cannot be obtained using any virtual switch on $G$. For example, the Silver-Williams representation $\varphi_{SW}:VB_n\to {\rm Aut}(F_{n}*\mathbb{Z}^{n+1})$ (see \cite{SilWil}), the Boden-Dies-Gaudreau-Gerlings-Harper-Nicas representation $\varphi_{BD}:VB_n\to {\rm Aut}(F_{n}*\mathbb{Z}^2)$ (see \cite{BDGGHN}),  the Kamada representation $\varphi_K:VB_n\to {\rm Aut}(F_{n}*\mathbb{Z}^{3n})$ (see \cite{BN}) and the representations $\varphi_M:VB_n\to {\rm Aut}(F_n*\mathbb{Z}^{2n+1})$, $\tilde{\varphi}_M:VB_n\to {\rm Aut}(F_n*\mathbb{Z}^n)$ of Bardakov-Mikhalchishina-Neshchadim (see \cite{BarMikNes, BarMikNes2}) cannot be obtained using any virtual switch. In order to overcome this obstacle, in the paper \cite{BarNas2} we introduce the notion of a (virtual) multi-switch which generalizes the notion of a (virtual) switch, and introduce a general approach how a virtual multi-switch $(S,V)$ on an algebraic system $X$ can be used for constructing a representation $\varphi_{S,V}:VB_n\to {\rm Aut}(X)$.  All the representations above can be obtained using this approach for appropriate virtual multi-switches. 

In the present paper we continue studying applications of virtual multi-switches. Namely, we introduce a general approach how virtual multi-switches can be used for constructing invariants of virtual links. For a given virtual multi-switch $(S,V)$ on an algebraic system $X$ (from some category) and a given virtual link $L$ we introduce the algebraic system $X_{S,V}(L)$ (from the same category as $X$) which is an invariant of $L$. As a corollary, we construct a new quandle invariant $\widetilde{Q}(L)$ for virtual links which generalizes the quandle of Manturov \cite{Man2}  and the quandle of Kauffman \cite{Kau}.

So, every multi-switch on an algebraic system leads to a virtual link invariant which is an algebraic system. The invariants which are algebraic systems are of special importance in the theory since usually they are very strong (for example, the knot quandle is a complete invariant for classical links under weak equivalence \cite{Joy,Mat}), and they lead to a lot of different invariants (see, for example, \cite{AHN,CSV,FMK}). 

The paper is organized as follows. In Section~\ref{knotandlinks}, we give necessary preliminaries about virtual  links. In Section~\ref{multnewmult}, we give preliminaries about multi-switches. In Section~\ref{secinv}, for a given multi-switch $(S,V)$ on an algebraic system $X$ and a given virtual link $L$ we introduce the algebraic system $X_{S,V}(L)$ and prove that $X_{S,V}(L)$ is a virtual link invariant (Theorem~\ref{ginvariant}). In Section~\ref{representationshelp}, we find a way how to construct the algebraic system $X_{S,V}(L)$ using the representation $\varphi_{S,V}:VB_n\to {\rm Aut}(X)$ introduced in \cite{BarNas2} (Theorem~\ref{usingrepr}). Finally, in Section~\ref{newquandlenorm}, we introduce a new quandle invariant for virtual links (Theorem~\ref{nquaninv2} and Theorem~\ref{nquaninv}).
\subsection*{Acknowledgement} The results are supported by the grant of the Russian Science Foundation (project 19-41-02005).

\section{Virtual knots and links}\label{knotandlinks}
In this section we recall definitions of virtual knots and links. Virtual links were introduced by Kauffman \cite{Kau} as a generalization of classical links. Topologically,
virtual links can be interpreted as isotopy classes of embeddings of classical links in thickened orientable
surfaces. \textit{A virtual $n$-component link diagram} is a regular immersion of $n$ oriented circles (i.~e. an immersion, such that there are no tangent components, and all intersection points are double points) such that every crossing is either classical (see Figure~\ref{possiblecrossings}~($a$, $b$)) or virtual (see Figure~\ref{possiblecrossings}~($c$)).
\begin{figure}[hbt!]
\noindent\centering{
\includegraphics[width=100mm]{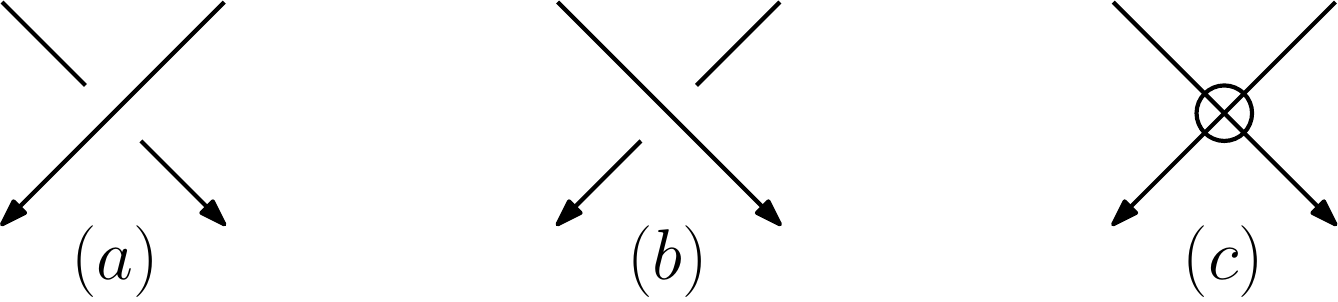}}
\caption{Crossings in the virtual link diagram.}
\label{possiblecrossings}
\end{figure}
The crossing depicted on Figure~\ref{possiblecrossings}~($a$) is called \textit{a positive crossing}, the crossing depicted on Figure~\ref{possiblecrossings}~($b$) is called \textit{a negative crossing}, and the crossing depicted on Figure~\ref{possiblecrossings}~($c$) is called \textit{a virtual crossing}. A virtual $1$-component link diagram is called \textit{a virtual knot diagram}. 

Two virtual link diagrams $D_1, D_2$ are said to be \textit{equivalent} if the diagram $D_1$ can be transformed to the diagram $D_2$ by planar isotopies and three types of
local moves (generalized Reidemeister moves): classical Reidemeister moves $R_1, R_2, R_3$ (see
Figure~\ref{rclass}), virtual Reidemeister moves $VR_1, VR_2, VR_3$ (see Figure~\ref{rvirt}), and mixed Reidemeister moves $VR_4$ (see Figure~\ref{rmix}). The equivalence class of a given virtual $n$-component link diagram is called \textit{a virtual $n$-component link}. A virtual $1$-component link is called \textit{a virtual knot}.
\begin{figure}[hbt!]
\noindent\centering{
\includegraphics[width=145mm]{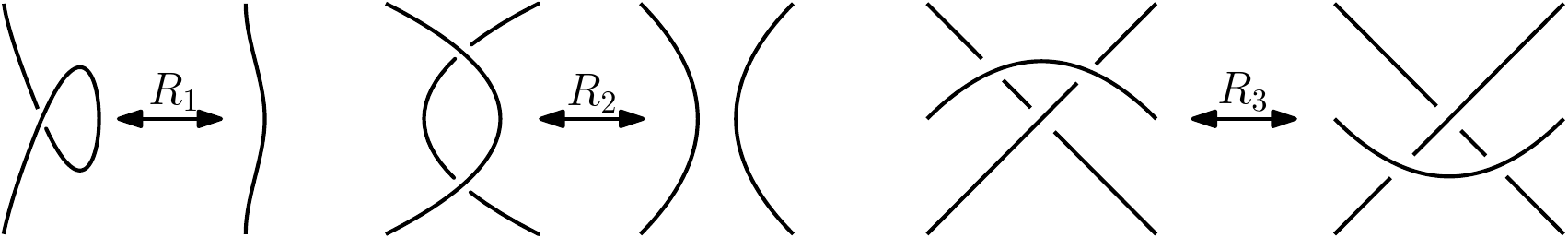}}
\caption{Classical Reidemeister moves.}
\label{rclass}
\end{figure}
\begin{figure}[hbt!]
\noindent\centering{
\includegraphics[width=145mm]{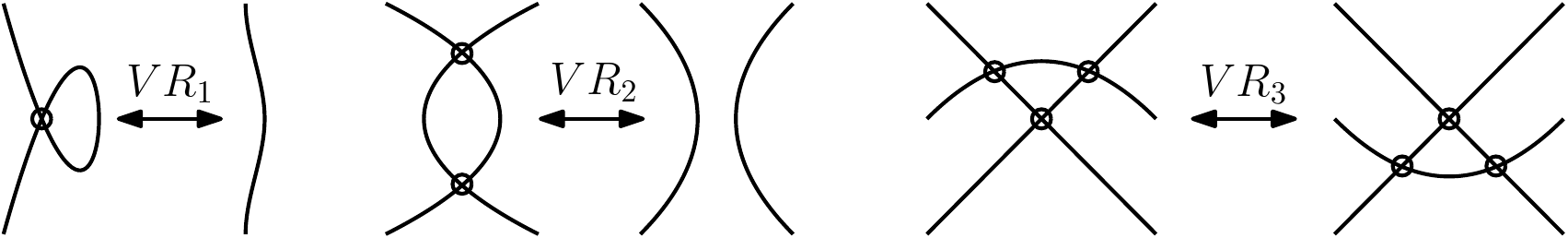}}
\caption{Virtual Reidemeister moves.}
\label{rvirt}
\end{figure}
\begin{figure}[hbt!]
\noindent\centering{
\includegraphics[width=62mm]{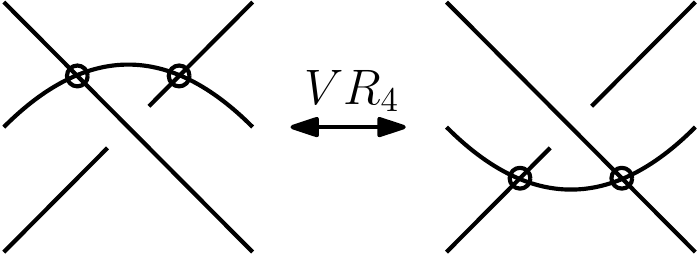}}
\caption{Mixed Reidemeister moves.}
\label{rmix}
\end{figure}

Let $A$ be an arbitrary non-empty set. The map $f$ from the set of all virtual link diagrams to $A$ is called \textit{an ($A$-valued) invariant} for virtual links if it maps equivalent link diagrams to the same element of $A$. Due to the definition of equivalent virtual link diagrams, the map $f$ is an invariant for virtual links if and only if $f(D_1)=f(D_2)$ whenever $D_1$ is obtained from $D_2$ using only one generalized Reidemeister move. If $f$ is an invariant for virtual links, and $L$ is a virtual link (i.~e. an equivalence class of some virtual link diagram $D$), then we write $f(L)=f(D)$.
\section{Switches and multi-switches}\label{multnewmult}
\subsection{Switches and the Yang-Baxter equation} \textit{A set-theoretical solution of the Yang-Baxter equation} is a pair $(X, S)$, where $X$ is a set, and $S:X^2\to X^2$ is a map such that
\begin{equation}\label{YB}(S\times id)(id \times S)(S\times id)=(id\times S)(S \times id)(id\times S).
\end{equation}
 If $(X,S)$ is a set theoretical solution of the Yang-Baxter equation such that the map $S:X^2\to X^2$ is bijective, then the map $S$ is called \textit{a switch} on $X$ (see \cite[Section~2]{FMK}). Let $S$ be a switch on $X$ such that  $S(x,y)=(S^l(x,y),S^r(x,y))$ for some maps $S^l, S^r:X^2\to X$ and  $x,y\in X$.  A switch $S$ is said to be \textit{non-degenerate} if for all elements $a,b\in X$ the maps $S^l_a, S^r_b:X\to X$ given by the formulas
\begin{align}
\notag S^l_a(x)=S^l(a,x),&&S^r_b(x)=S^r(x,b)
\end{align}
for $x\in X$ are invertible.  A switch $S$ is called \textit{involutive} if $S^2=id$.

 If $X$ is not just a set but an algebraic system: a group, a module etc, then every switch $S$ on $X$ is called, respectively, a group switch, a module switch etc. 
\begin{example}\label{twist}
 Let $X$ be an arbitrary set, and $T(a,b)=(b,a)$ for all $a, b\in X$. The map $T$ is a switch which is called \textit{the twist}.
 It is clear that $T$ is involutive and non-degenerate.
\end{example}
\begin{example}\label{Artin switch}
Let $X$ be a group. Then the map $S_A$ defined by $
S_A(a, b) = (a b a^{-1}, a)$ for $a,b\in X$ is a switch which is called \textit{the Artin switch} on a group $X$. This switch is non-degenerate.
\end{example}
In order to introduce the next example of a switch, let us recall that \textit{a quandle} $Q$ is an algebraic system with one binary algebraic operation $(a,b)\mapsto a*b$ which satisfies the following three axioms.
\begin{enumerate}
	\item For all $a\in Q$ the equality $a*a=a$ holds.
	\item The map $I_a:b\mapsto b*a$ is a bijection of $Q$ for all $a\in Q$. For $a,b\in Q$ we denote by $a*^{-1}b=I_b^{-1}(a)$. 
	\item For all $a,b,c\in Q$ the equality $(a*{b})*c=(a*c)*(b*c)$ holds.
\end{enumerate}
A quandle $Q$ is called trivial if $a*b=a$ for all $a,b\in Q$, the trivial quandle with $n$ elements is denoted by $T_n$. 

Quandles were introduced in \cite{Joy, Mat} as an invariant for classical links. More precisely,  to each oriented diagram $D$ of an oriented knot $K$ in $\mathbb{R}^3$ one can associate the quandle $Q(K)$ which does not change if we apply the Reidemeister moves to the diagram $D$. Kauffman \cite{Kau} extended the idea of the quandle invariant $Q(K)$ from classical links to virtual links by ignoring virtual crossings (ignoring virtual crossings makes the quandle invariant for virtual links weaker than for classical links). Manturov \cite{Man2}  constructed a quandle invariant for virtual links which generalizes the quandle of Kauffman \cite{Kau}. Over the years, quandles have been investigated by various authors for constructing new invariants for knots and links (see, for example, \cite{Carter, CatNas, Fenn-Rourke, Kamada20122, NanSinSin, Nelson}). Algebraic properties of quandles including their automorphisms and residual properties have been investigated, for example, in \cite{BDS, BNS, BSS, BiaBon, Clark, NelWon}. For more details about quandles see \cite{Carter, ElhNel, Nos}.

The following  example gives a quandle switch.
\begin{example}\label{quasw} Let $(X,*)$ be a quandle. Then the map $S_Q(a, b) = (b, a*b)$ for $a, b \in X$ is a quandle switch. This switch is non-degenerate.
\end{example}

A lot of examples of switches on different algebraic systems can be found, for example, in \cite[Section~2]{BarNas2}, \cite[Section~3]{FB0} and \cite[Section~2]{FMK}. 

\subsection{Switches and invariants of virtual links}\label{swinv} Switches can be used for constructing invariants for virtual links. Let $S$ be a switch on $X$ such that
$$S(a,b)=(S^l(a,b), S^r(a,b))$$
for maps $S^l, S^r:X^2\to X$ and $a,b\in X$. For $a,b\in X$ denote by  $S^l(a,b)=b_a$, $S^r(a,b)=a^b$, so, on $X$ we have two binary operations $(a, b) \mapsto a^b$, $(a, b) \mapsto a_b$, which are called \textit{the up operation} and \textit{the down operation} defined by $S$, respectively. The Yang-Baxter equation for $S$ implies the following equalities
\begin{align}
\label{biqiq} a^{bc}=a^{c_bb^c},&&a_{bc}=a_{c^bb_c},&&(a_b)^{c_{b^a}}=(a^{c})_{b^{c_a}}
\end{align}
for all $a,b,c\in X$.  A switch $S$ is called \textit{a biquandle switch} on $X$ if the following two conditions hold.
\begin{enumerate}
\item The switch $S$ is non-degenerate, i.~e. the maps $S^l_a,S^r_a:X\to X$ given by  $S^l_a(x)=x_a$, $S^r_a(x)=x^a$ are bijective. We denote by 
\begin{align}
\notag b^{a^{-1}}=(S^r_a)^{-1}(b),&& b_{a^{-1}}=(S^l_a)^{-1}(b).
\end{align} 
\item $a^{a^{-1}}=a_{a^{a^{-1}}}$ and $a_{a^{-1}}=a^{a_{a^{-1}}}$ for all $a\in X$.
\end{enumerate}
The switches from Examples~\ref{twist}, \ref{Artin switch}, \ref{quasw} are biquandle switches. If $S$ is a biquandle switch on $X$, then the set $X$ with the up and the down operations defined by $S$ is called \textit{a biquandle} and is denoted by $(X,S)$. If condition (2) of the biquandle switch does not hold, then the algebraic system $(X,S)$ is called \textit{a birack}.

Biquandles were introduced in \cite{FMK} as a tool for constructing invariants for virtual knots and links. Let $S$ be a biquandle switch on $X$ with 
$$S(a,b)=(S^l(a,b), S^r(a,b))$$ 
for maps $S^l, S^r:X^2\to X$, and $a,b\in X$. Since $S$ is a biquandle switch on $X$, the algebraic system $(X,S)$ is a biquandle. Let the inverse to $S$ map $S^{-1}:X^2\to X^2$ has the form 
$$S^{-1}(a,b)=(Q^l(a,b), Q^r(a,b))$$
for maps $Q^l, Q^r:X^2\to X$, and $a,b\in X$.

Let $D$ be a virtual link diagram. A strand of $D$ going from one crossing (classical or virtual) to another crossing (classical or virtual) is called \textit{an arc} of $D$. \textit{A labelling} of the diagram $D$ by elements of the biquandle $(X,S)$ is a marking of arcs of $D$ by elements of $X$ such that near each crossing of $D$  the marks of arcs are as on Figure~\ref{biqrel}. 
\begin{figure}[hbt!]
\noindent\centering{
\includegraphics[width=120mm]{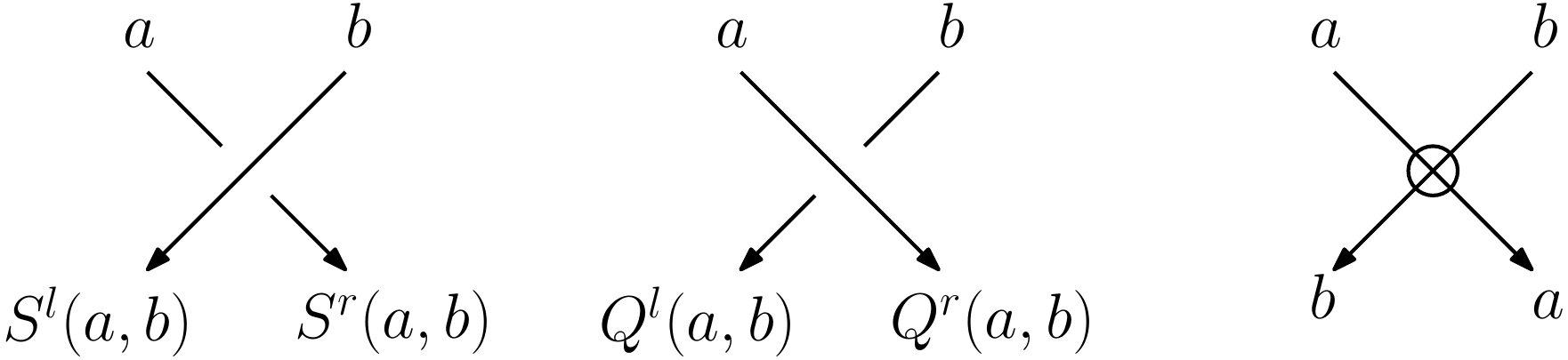}}
\caption{Marks of arcs in $D$.}
\label{biqrel}
\end{figure}
A diagram $D$ can have zero, one or several (even infinitely many) labellings by elements of the biquandle $(X,S)$. The set of all labellings of the diagram $D$ by elements of the biquandle $(X,S)$ is denoted by ${\rm lab}_{(X,S)}(D)$. If $X$ is a finite set, then the set ${\rm lab}_{(X,S)}(D)$ is finite.

Fenn, Jordan-Santana and Kauffman in \cite[Theorem~6.12]{FMK} proved that if $D_1$, $D_2$ are equivalent virtual link diagrams, then each labelling of the diagram $D_1$ provides a unique labelling of the diagram $D_2$. In particular, if $X$ is a finite set, then $\left|{\rm lab}_{(X,S)}(D_1)\right|=\left|{\rm lab}_{(X,S)}(D_2)\right|$, i.~e. the integer $\left|{\rm lab}_{(X,S)}(D)\right|$ is an invariant for virtual links. This invariant is called \textit{the biquandle-coloring invariant} defined by $(X,S)$ and is denoted by $C_{(X,S)}(D)$. So, every biquandle switch provides an invariant for virtual links. Papers \cite{CSWES, KM, F, HK, FB, BarNasSin14} give several application of biquandles in knot theory.

\subsection{Multi-switches and virtual multi-switches} Let $X$ be a non empty set, and $X_1, X_2, \dots, X_m$ be non-empty subsets of $X$. One can identify the sets
\begin{align}
\notag (X\times X_1\times X_2\times\dots\times X_m)^2&&\text{and}&&X^2\times X_1^2\times X_2^2\times\dots\times X_m^2
\end{align}
and for elements $A=(a_0,a_1,\dots,a_m), B=(b_0,b_1,\dots,b_m)$ from $X\times X_1\times X_2\times\dots\times X_m$ denote the ordered pair $(A,B)$ by 
$$(A,B)=(a_0,b_0; a_1,b_1;a_2,b_2;\dots;a_m,b_m).$$

We say that a map $S:X^2\times X_1^2\times \dots\times X_m^2\to X^2\times X_1^2\times \dots\times X_m^2$ is \textit{an $(m+1)$-switch}, or \textit{a multi-switch} on $X$ (if $m$ is not specified) if $S$ is a switch on $X\times X_1\times X_2\times\dots\times X_m$, such that
$$S(c_0,c_1,\dots,c_m)=(S_0(c_0,c_1,\dots,c_m),S_1(c_1),\dots,S_m(c_m))$$
for $c_0\in X^2$, $c_i\in X_i^2$ for $i=1,2,\dots,m$,
and $S_0, S_1,\dots,S_m$ are the maps
\begin{align}
\notag S_0&:X^2 \times X_1^2 \times \dots \times X_m^2 \to X^2,\\
\notag S_i&:X_i^2  \to X_i^2,~\text{for}~i = 1, 2, \dots, m.
\end{align}
If $S$ is an $(m+1)$-switch on $X$ defined by the maps $S_0,S_1,\dots,S_m$, then we write $S=(S_0,S_1,\dots,S_m)$. Note that for $i=1,2,\dots,m$ the map $S_i$ is a switch on $X_i$.  If $X$ is not just a set but an algebraic system: group, module etc, then every multi-switch on $X$ is called, respectively, a group multi-switch, a module multi-switch etc. We do not require here that $X_1, X_2, \dots,X_m$ are subsystems of $X$.
\begin{example}\label{exx7}Every switch on $X$ is a $1$-switch on $X$.
\end{example}
\begin{example}\label{exx8}If $S$ is a switch on $X$, and  $S_i$ is a switch on $X_i\subset X$ for $i=1,2,\dots,m$, then the map $S\times S_1\times \dots\times S_m$ is an $(m+1)$-switch on $X$.
\end{example}
The multi-switches from Example~\ref{exx7} and Example~\ref{exx8} are in some sense trivial. The following example introduced in \cite[Proposition~1]{BarNas2} gives a non-trivial example of a module $2$-switch.
\begin{example}Let $R$ be an integral domain, $X$ be a free module over $R$, and $X_1$ be a subset of the multiplicative group of $R$ (one can think about $X_1$ as about a subset of $X$ identifying $X_1\subset R$ with  $X_1x_0\subset X$ for a fixed non-zero element $x_0$ from $X$). Then the map $S:X^2\times X_1^2\to X^2\times X_1^2$ given by
\begin{align}
\notag S(a, b; x, y) = ((1 - y) a + x b, a; y, x)&&a,b\in X,~x,y\in X_1
\end{align}
is a $2$-switch on $X$.
\end{example}

Since an $(m+1)$-switch on $X$ is a switch on $X\times X_1\times X_2\times\dots\times X_m$, the notion of the involutive $(m+1)$-switch follows from the same notion for switches. Let $S, V: X^2\times X_1^2\times X_2^2\times\dots\times X_m^2\to X^2\times X_1^2\times X_2^2\times\dots\times X_m^2$ be an $(m+1)$-switch and an involutive $(m+1)$-switch on $X$, respectively. We say that the pair $(S,V)$ is \textit{a virtual $(m+1)$-switch} on $X$ (or a \textit{virtual multi-switch} on $X$, if $m$ is not specified) if the following equality holds
\begin{equation}\label{VYB}
 (id \times V) (S \times id)(id \times V)  =   (V \times id)(id \times S)(V \times id),
\end{equation}
where $id$ denotes the identity on $X\times X_1\times X_2\times\dots\times X_m$. Some examples of virtual multi-switches can be found in \cite{BarNas2}.

\section{Multi-switches and knot invariants}\label{secinv}
In Section~\ref{swinv} we noticed that biquandle switches can be used for constructing invariants for virtual links. Namely, if $S$ is a biquandle switch on a finite set $X$, then the number $C_{(X,S)}(D)$ of labellings of a virtual link diagram $D$ by elements of the biquandle $X$ is an integer-valued invariant for virtual links. 
 In this section we introduce a general construction how a given multi-switch on an algebraic system $X$ can be used for constructing an algebraic system which is an invariant for virtual links. Using this approach it is possible to construct a group of a link, a quandle of a link, a biquandle of a link,  a module of a link, a skew brace of a link (see \cite{GuaVen,Nas, SmoVen} for details about skew braces) etc.

  Let $X$ be an algebraic system, and $X_0,X_1,\dots, X_m$ be subsystems of $X$ such that
\begin{enumerate}
\item   for $i=0,1,\dots,m$ the subsystem $X_i$ is generated by elements $x^{i}_{1},x^{i}_{2},\dots$ (which are not necessarily all different),
\item $\{x^{i}_{1},x^{i}_{2},\dots\}\cap\{x^{j}_{1},x^j_2,\dots\}=\varnothing$ for $i\neq j$,
\item the set of elements $\{x^{i}_j~|~i=0,1,\dots,m, j=0,1,\dots\}$ generates $X$,
\item for every permutation $\alpha$ of $\mathbb{N}$ with a finite support the map $x^i_j\mapsto x^i_{\alpha(j)}$ for $i=0,1,\dots,m$, $j=1,2,\dots$ induces an automorphism of $X$.
\end{enumerate}
\begin{remark}In this section we require that $X_1, X_2, \dots,X_m$ are subsystems of $X$, while in the definition of the virtual multi-switch we do not require this condition. 
\end{remark}
\begin{remark}Condition (4) implies that for a fixed $i$ the elements $x^i_1, x^i_2,\dots$ are either all different, or all coincide.
\end{remark}
 For $n\geq 2$, $i=0,1,\dots,m$ denote by $X_i^{(n)}$ the subsystem of $X_i$ generated by $x^i_1,x^i_2,\dots,x^i_n$, and by $X^{(n)}$ the subsystem of $X$ generated by $X_0^{(n)}, X_1^{(n)},\dots,X_m^{(n)}$.
\begin{remark}\label{remark}
From condition (4) follows that for $n>1$ the algebraic system $X^{(n)}$ is isomorphic to the quotient of $X^{(n+1)}$ by $(m+1)$ relations 
\begin{align}\label{alot of relations in one}
 x_{n+1}^0=x_n^0,&& x_{n+1}^1=x_n^1,&&\dots&&x_{n+1}^m=x_n^m
\end{align}
(we add to $X^{(n)}$ generators $x_{n+1}^0,x_{n+1}^1,\dots,x_{n+1}^m$ and say that they are equal to $x_{n}^0, x_{n}^1, \dots, x_n^m$). We will write $(m+1)$ equalities from (\ref{alot of relations in one}) as one equality of $(m+1)$-tuples
$$(x_{n+1}^0,x_{n+1}^1,\dots,x_{n+1}^m)=(x_{n}^0, x_{n}^1, \dots, x_n^m).$$
\end{remark}

Let $S=(S_0,S_1,\dots,S_m)$, $V=(V_0,V_1,\dots,V_m)$  be a virtual $(m+1)$-switch on $X$ such that
\begin{align}
\notag S_0=(S_0^l,S_0^r), V_0=(V_0^l,V_0^r)&:X^2 \times X_1^2 \times X_2^2 \times \dots \times X_m^2 \to X^2,\\
\notag S_i=(S_i^l,S_i^r), V_i=(V_i^l,V_i^r)&:X_i^2  \to X_i^2,~\text{for}~i = 1, 2, \dots, m,
\end{align}
and for $i=0,1,\dots,m$ the images of maps $S_i^l,S_i^r, V_i^l, V_i^r$  are words over its arguments in terms of the operations of $X$. From this fact follows, in particular, that $(S,V)$ induces a virtual $(m+1)$-switch on $X^{(n)}$ for all $n$. Since the map
$$S:X^2\times X_1^2\times X_2^2\dots\times X_m^2\to X^2\times X_1^2\times X_2^2\dots\times X_m^2$$
is a switch, it is invertible, and we can speak about $S^{-1}$.

Let $D$ be a virtual link diagram with $n$ arcs.  Mark the arcs of $D$ by $(m+1)$-tuples $\widetilde{x}_j=(x^0_j, x^1_j,\dots,x^m_j)$ for $j=1,2,\dots,n$, where  $x^0_j, x^1_j,\dots,x^m_j$ are generators of $X$ described at the beginning of this section.  Let the labels of arcs near some crossing be as on Figure~\ref{newconlabels} (the big circle in the middle of the crossing means that the crossing can be positive, negative, or virtual).
\begin{figure}[hbt!]
\noindent\centering{
\includegraphics[width=30mm]{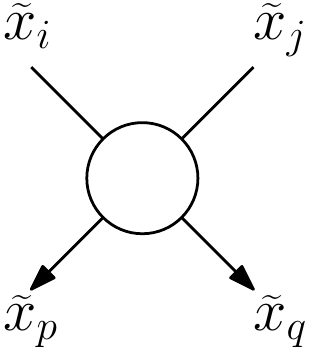}}
\caption{Labels of arcs of $D$ near some crossing.}
\label{newconlabels}
\end{figure}

By the definition, a subsystem of $X$ generated by all $x_j^{i}$ from labels on arcs of $D$ is $X^{(n)}$. Denote by $X_{S,V}(D)$ the quotient of $X^{(n)}$ by the relations which can be written from the crossings of $D$ in the following way:
\begin{align}
\notag S(\widetilde{x}_i,\widetilde{x}_j)&=(\widetilde{x}_p,\widetilde{x}_q),&&\text{if the crossing is positive},\\
\notag S^{-1}(\widetilde{x}_i,\widetilde{x}_j)&=(\widetilde{x}_p,\widetilde{x}_q),&&\text{if the crossing is negative},\\
\notag V(\widetilde{x}_i,\widetilde{x}_j)&=(\widetilde{x}_p,\widetilde{x}_q),&&\text{if the crossing is virtual},
\end{align}
where the labels are as on Figure~\ref{newconlabels}. Note that each crossing of $D$ gives $2(m+1)$ relations, so, the number of relations is equal to the number of crossings multiplied by $2(m+1)$.
\begin{theorem}\label{ginvariant} Let $X$ be an algebraic system, and $X_0,X_1,\dots,X_m$ be subsystems of $X$ which satisfy conditions (1)-(4) from the beginning of this section. If $(S,V)$ is a virtual $(m+1)$-switch on $X$ such that the maps $S,V$ are biquandle switches on $X\times X_1\times X_2\times \dots\times X_m$, then $X_{S,V}(D)$ is an invariant for virtual links.
\end{theorem}
\begin{proof} Due to condition (4) from the beginning of this section, for every permutation $\alpha$ of $\mathbb{N}$ with a finite support the map $x_j^{i}\mapsto x^i_{\alpha(j)}$ for $i=0,1,\dots,m$, $j=1,2,\dots$ induces an automorphism of $X$. From this condition follows that the algebraic system $X_{S,V}(D)$ is well defined, i.~e. it is not important how (in which sequence) we label the arcs in $D$.

In order to prove that $X_{S,V}$ is an invariant for virtual links, it is enough to prove that $X_{S,V}(D_1)$ is isomorphic to $X_{S,V}(D_2)$ in the case when $D_2$ is obtained from $D_1$ using only one generalized Reidemeister move (see Figures~\ref{rclass}, \ref{rvirt}, \ref{rmix}). Depending on this generalized Reidemeister move we consider several cases.

\textit{Case 1: $D_2$ is obtained from $D_1$ by $R_1$-move or $VR_1$-move.} We consider in details the case of the move $R_1$. The diagrams $D_1$ and $D_2$ differ only in the small neighborhood, where $D_1$ is as on the left part of Figure~\ref{r1difference}, and $D_2$ is as on the right part of Figure~\ref{r1difference}.
\begin{figure}[hbt!]
\noindent\centering{
\includegraphics[width=25mm]{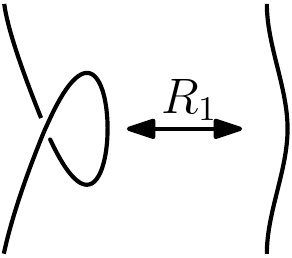}}
\caption{The neighborhood, where $D_1$ and $D_2$ are different: $D_1$ is on the left, $D_2$ is on the right.}
\label{r1difference}
\end{figure}

Suppose that outside of the neighborhood where the $R_1$-move is applied diagrams $D_1$ and $D_2$ have $n$ arcs. Label the arcs of $D_1$ and $D_2$ (which are not labeled yet) inside of the neighborhood where the $R_1$-move is applied by additional labels. Then in the neighbourhood where we apply the $R_1$-move the arcs of the diagram $D_2$ have the labels depicted on Figure~\ref{r12}.
\begin{figure}[hbt!]
\noindent\centering{
\includegraphics[width=5mm]{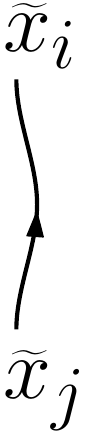}}
\caption{Labels of the arcs in $D_2$.}
\label{r12}
\end{figure}
Therefore $X_{S,V}(D_2)$ is the quotient of $X^{(n)}$ by the relations which can be written from the part of the diagram outside of Figure~\ref{r12} and $m+1$ relations
$$\widetilde{x}_i=\widetilde{x}_j$$
from the part of the diagram on Figure~\ref{r12}. 

Let us find $X_{S,V}(D_1)$. In the neighbourhood where we apply the $R_1$-move the arcs of the diagram $D_1$ have the labels depicted on Figure~\ref{r11}. 
\begin{figure}[hbt!]
\noindent\centering{
\includegraphics[width=20mm]{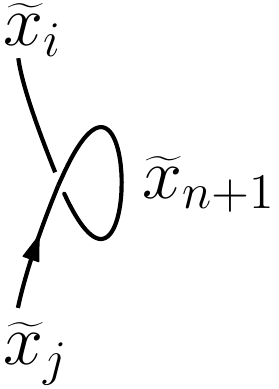}}
\caption{Labels of the arcs in $D_1$.}
\label{r11}
\end{figure}
By the definition the algebraic system $X_{S,V}(D_1)$ is the quotient of $X^{(n+1)}$ ($n$ arcs outside of Figure~\ref{r11} and one extra arc inside of Figure~\ref{r11}) by the relations which can be written from the part of the diagram outside of Figure~\ref{r11} and $2(m+1)$ relations
\begin{equation}\label{eqr1}
S(\widetilde{x}_{n+1}, \widetilde{x}_j)=(\widetilde{x}_{n+1},\widetilde{x}_i)
\end{equation}
which are written from the crossing on Figure~\ref{r11}. For the elements $A,B$ from the set $X\times X_1\times X_2\times\dots\times X_m$ denote by $S(A,B)=(B_A,A^B)$, where $B_A,A^B$ are the elements from  $X\times X_1\times \dots\times X_m$.  Since $S$ is a biquandle switch on $X\times X_1\times X_2\times\dots\times X_m$, the maps $B\mapsto B_A$, $B\mapsto B^A$ are invertible and we can speak about $B_{A^{-1}}$, $B^{A^{-1}}$ for $A,B\in X\times X_1\times \dots\times X_m$. In these denotations equalities (\ref{eqr1}) can be written in the form
\begin{align}
\label{cord1} (\tilde{x}_j)_{\tilde{x}_{n+1}}&=\tilde{x}_{n+1},\\
\label{cord2} (\tilde{x}_{n+1})^{\tilde{x}_j}&=\tilde{x}_i.
\end{align}
From equality~(\ref{cord2}) follows that $\tilde{x}_{n+1}=(\tilde{x}_i)^{\tilde{x}_j^{-1}}$. Since the images of the maps $S_k^l,S_k^r$  are words over its arguments in terms of operations of $X$ for $k=0,1,\dots,m$, the equality $\tilde{x}_{n+1}=(\tilde{x}_i)^{\tilde{x}_j^{-1}}$ implies that the generators $x_{n+1}^0,x_{n+1}^1,\dots,x_{n+1}^m$ can be expressed as words over generators $x_{i}^0,x_{i}^1,\dots,x_{i}^m$, $x_{j}^0,x_{j}^1,\dots,x_{j}^m$. So, we can delete the elements $x_{n+1}^0,x_{n+1}^1,\dots,x_{n+1}^m$ from the generating set of $X_{S,V}(D_1)$.  Therefore $X_{S,V}(D_1)$ is the quotient of $X^{(n)}$ (we change $n+1$ by $n$ since we deleted all the elements $x_{n+1}^0,x_{n+1}^1,\dots,x_{n+1}^m$ from the generating set) by the relations which can be written from the part of the diagram outside of Figure~\ref{r11} and $m+1$ relations which are obtained from equalities (\ref{cord1}), (\ref{cord2}) excluding $\tilde{x}_{n+1}$. From (\ref{cord1}) we have the equality 
\begin{align}\label{cord1new}
\widetilde{x}_j=(\widetilde{x}_{n+1})_{\widetilde{x}_{n+1}^{-1}},
\end{align}
from (\ref{cord2}) we have the equality
 \begin{align}\label{cord2new}
\widetilde{x}_i=(\widetilde{x}_{n+1})^{\widetilde{x}_{j}}=(\widetilde{x}_{n+1})^{(\widetilde{x}_{n+1})_{\widetilde{x}_{n+1}^{-1}}},
\end{align} 
and using the second axiom of biquandles we can exlude $\widetilde{x}_{n+1}$ from  equalities (\ref{cord1new}), (\ref{cord2new}) and obtain the equality  $\tilde{x}_i=\tilde{x}_j$. Therefore $X_{S,V}(D_1)$ is the quotient of $X^{(n)}$  by the relations which can be written from the part of the diagram outside of Figure~\ref{r11} and $m+1$ relations
\begin{equation}\label{i=j}\tilde{x}_i=\tilde{x}_j.
\end{equation}
Comparing this description of $X_{S,V}(D_1)$ with the description of $X_{S,V}(D_2)$, we see that $X_{S,V}(D_1)=X_{S,V}(D_2)$.

The case of the $R_1$-move with the negative crossing, and the case of the  $VR_1$-move are almost the same with small mutations.

\textit{Case 2: $D_2$ is obtained from $D_1$ by $R_2$-move or $VR_2$-move.}
We consider in details only the case of the move $R_2$, the case of the move $VR_2$ is similar. The diagrams $D_1$ and $D_2$ differ only in the small neighborhood, where $D_1$ is as on the left part of Figure~\ref{r2difference}, and $D_2$ is as on the right part of Figure~\ref{r2difference}.
\begin{figure}[hbt!]
\noindent\centering{
\includegraphics[width=40mm]{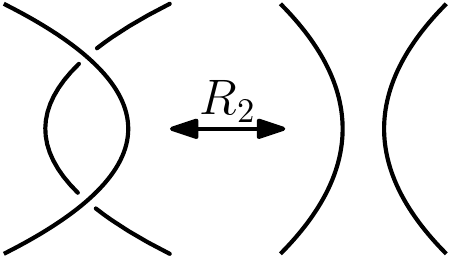}}
\caption{The neighborhood, where $D_1$ and $D_2$ are different: $D_1$ is on the left, $D_2$ is on the right.}
\label{r2difference}
\end{figure}

Suppose that outside of the neighborhood where the $R_2$-move is applied diagrams $D_1$ and $D_2$ have $n$ arcs. Label the arcs of $D_1$ and $D_2$ (which are not labeled yet) inside of the neighborhood where the $R_2$-move is applied by additional labels. Then in the neighbourhood where we apply the $R_2$-move the arcs of the diagram $D_2$ have the labels depicted on Figure~\ref{r22}.
\begin{figure}[hbt!]
\noindent\centering{
\includegraphics[width=20mm]{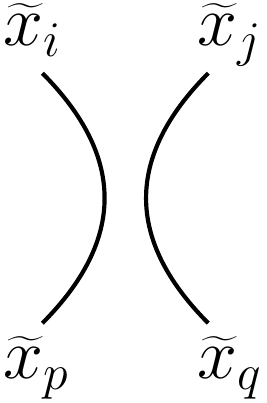}}
\caption{Labels of the arcs in $D_2$.}
\label{r22}
\end{figure}
Therefore $X_{S,V}(D_2)$ is the quotient of $X^{(n)}$ by the relations which can be written from the part of the diagram outside of Figure~\ref{r22} and $2(m+1)$ relations
\begin{align}
\notag \widetilde{x}_i=\widetilde{x}_p,&&\widetilde{x}_j=\widetilde{x}_q
\end{align}
from the part of the diagram on Figure~\ref{r22}.

Let us find $X_{S,V}(D_1)$. In the neighborhood where we apply the $R_2$-move the arcs of the diagram $D_1$ have the labels depicted on Figure~\ref{r21}.
  \begin{figure}[hbt!]
\noindent\centering{
\includegraphics[width=28mm]{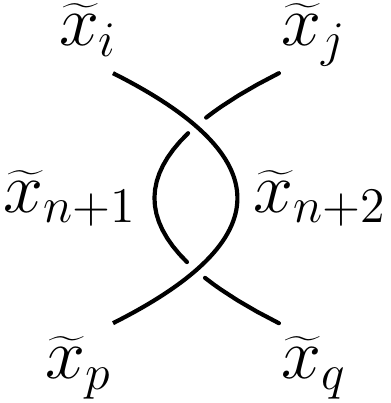}}
\caption{Labels of the arcs in $D_1$.}
\label{r21}
\end{figure}
There are four possibilities for the orientation of the arcs on Figure~\ref{r21}: i) both arcs are oriented from up to down, ii) both arcs are oriented from down to up, iii) the left arc is oriented from up to down, and the right arc is oriented from down to up, and iv) the left arc is oriented from down to up, and the right arc is oriented from up to down. Since case i) is  very similar to case ii), and case iii) is very similar to case iv), we will consider in details only case i) and case iii).

Let both arcs on Figure~\ref{r21} are oriented from the top to the bottom. In this case $X_{S,V}(D_1)$ is the quotient of $X^{(n+2)}$ ($n$ arcs outside of Figure~\ref{r21} and $2$ new arcs inside of Figure~\ref{r21}) by the relations which can be written from the part of the diagram outside of Figure~\ref{r21} and $4(m+1)$ relations
\begin{align}
\label{r2cros1}S^{-1}(\widetilde{x}_i,\widetilde{x}_j)&=(\widetilde{x}_{n+1}, \widetilde{x}_{n+2}),\\
\label{r2cros2}S(\widetilde{x}_{n+1}, \widetilde{x}_{n+2})&=(\widetilde{x}_{p}, \widetilde{x}_q)
\end{align}
which are written from two crossings on Figure~\ref{r21}.

Since for $k=0,1,\dots, m$ the images of the maps $S_k^l,S_k^r$  are words over its arguments in terms of operations of $X$, from equality~(\ref{r2cros1}) follows that the generators $x_{n+1}^0,x_{n+1}^1,\dots,x_{n+1}^m$, $x_{n+2}^0,x_{n+2}^1,\dots,x_{n+2}^m$ can be expressed as words over generators $x_{i}^0,x_{i}^1,\dots,x_{i}^m$, $x_{j}^0,x_{j}^1,\dots,x_{j}^m$. Hence, we can delete the elements $x_{n+1}^0,x_{n+1}^1,\dots,x_{n+1}^m$, $x_{n+2}^0,x_{n+2}^1,\dots,x_{n+2}^m$ from the generating set of $X_{S,V}(D_1)$. Therefore $X_{S,V}(D_1)$ is the quotient of $X^{(n)}$ (here we change $n+2$ by $n$ since we deleted all the elements $x_{n+1}^0,x_{n+1}^1,\dots,x_{n+1}^m$, $x_{n+2}^0,x_{n+2}^1,\dots,x_{n+2}^m$ from the generating set) by the relations which can be written from the part of the diagram outside of Figure~\ref{r21} and $2(m+1)$ relations obtained from equality (\ref{r2cros2}) changing $(\widetilde{x}_{n+1},\widetilde{x}_{n+2})$ by $S^{-1}(\widetilde{x}_i,\widetilde{x}_j)$  (that is what we have from equality (\ref{r2cros1})). This equality clearly has the form
\begin{equation}\label{2i=2j}
(\widetilde{x}_i,\widetilde{x}_j)=(\widetilde{x}_p,\widetilde{x}_q).
\end{equation}
Therefore, we proved that if both arcs on Figure~\ref{r21} are oriented from the top to the bottom, then  $X_{S,V}(D_1)$ is the quotient of $X^{(n)}$  by the relations which can be written from the part of the diagram outside of Figure~\ref{r21} and $2(m+1)$ relations~(\ref{2i=2j}). Comparing this description of $X_{S,V}(D_1)$ with the description of $X_{S,V}(D_2)$, we see that $X_{S,V}(D_1)=X_{S,V}(D_2)$.

Let the left arc on Figure~\ref{r21} is oriented from up to down, and the right arc on Figure~\ref{r21} is oriented from down to up. In this case $X_{S,V}(D_1)$ is the quotient of $X^{(n+2)}$ by the relations which can be written from the part of the diagram outside of Figure~\ref{r21} and $4(m+1)$ relations
\begin{align}
\label{r22cros1}S(\widetilde{x}_{n+1},\widetilde{x}_i)&=(\widetilde{x}_{n+2}, \widetilde{x}_j),\\
\label{r22cros2}S^{-1}(\widetilde{x}_{n+2}^0,\widetilde{x}_{q}^0)&=(\widetilde{x}_{n+1}^0,\widetilde{x}_{p})
\end{align}
which are written from two crossings depicted on Figure~\ref{r21}.  For elements $A,B$ from  $X\times X_1\times X_2\times\dots\times X_m$ denote by $S(A,B)=(B_A,A^B)$, where $B_A,A^B$ are the elements from  $X\times X_1\times \dots\times X_m$.  Since $S$ is a biquandle switch on $X\times X_1\times X_2\times\dots\times X_m$, the maps $B\mapsto B_A$, $B\mapsto B^A$ are invertible and we can speak about $B_{A^{-1}}$, $B^{A^{-1}}$ for $A,B\in X\times X_1\times \dots\times X_m$. In these denotations equalities (\ref{r22cros1}), (\ref{r22cros2}) can be written in the form
\begin{align}
\label{almost1} (\tilde{x}_i)_{\tilde{x}_{n+1}}&=\tilde{x}_{n+2},\\
\label{almost2} (\tilde{x}_{n+1})^{\tilde{x}_i}&=\tilde{x}_j,\\
\label{almost3} (\tilde{x}_p)_{\tilde{x}_{n+1}}&=\tilde{x}_{n+2},\\
\label{almost4} (\tilde{x}_{n+1})^{\tilde{x}_p}&=\tilde{x}_q.
\end{align}
Since for $k=0,1,\dots,m$ the images of the maps $S_k^l,S_k^r$  are words over its arguments in terms of operations of $X$, from equalities~(\ref{almost1}), (\ref{almost2}) follows that
\begin{align}
\notag\tilde{x}_{n+1}=(\tilde{x}_j)^{\tilde{x}_i^{-1}},&&\tilde{x}_{n+2}=(\tilde{x}_i)_{\tilde{x}_{n+1}}=(\tilde{x}_i)_{(\tilde{x}_j)^{\tilde{x}_i^{-1}}},
\end{align}
i.~e. the generators $x_{n+1}^0,x_{n+1}^1,\dots,x_{n+1}^m$, $x_{n+2}^0,x_{n+2}^1, \dots,x_{n+2}^m$ can be expressed as words over generators $x_{i}^0,x_{i}^1,\dots,x_{i}^m$, $x_{j}^0,x_{j}^1,\dots,x_{j}^m$.  So, we can delete the elements $x_{n+1}^0,x_{n+1}^1,\dots,x_{n+1}^m$, $x_{n+2}^0,x_{n+2}^1,\dots,x_{n+2}^m$ from the generating set of $X_{S,V}(D_1)$. Therefore $X_{S,V}(D_1)$ is the quotient of $X^{(n)}$ (we changed $n+2$ by $n$ since we deleted the elements $x_{n+1}^0,x_{n+1}^1,\dots,x_{n+1}^m$, $x_{n+2}^0,x_{n+2}^1,\dots,x_{n+2}^m$ from the generating set) by the relations which can be written from the part of the diagram outside of Figure~\ref{r21} and relations obtained from equalities (\ref{almost3}), (\ref{almost4}) changing $\tilde{x}_{n+1}$ by $(\tilde{x}_j)^{\tilde{x}_i^{-1}}$, and changing $\tilde{x}_{n+2}$ by $(\tilde{x}_i)_{(\tilde{x}_j)^{\tilde{x}_i^{-1}}}$ (that is what we have from equalities (\ref{almost1}), (\ref{almost2})). It is easy to see that these equalities will have the form
 \begin{align}
\label{3i=3j} \tilde{x}_i=\tilde{x}_p,&&\tilde{x}_j=\tilde{x}_q.
 \end{align}
Comparing this description of $X_{S,V}(D_1)$ with the description of $X_{S,V}(D_2)$, we see that $X_{S,V}(D_1)=X_{S,V}(D_2)$ in the case when the left arc on Figure~\ref{r21} is oriented from up to down, and the right arc on Figure~\ref{r21} is oriented from down to up.

\textit{Case 3: $D_2$ is obtained from $D_1$ by $R_3$-move, $VR_3$-move, or $VR_4$-move.} We consider in details only the case of the move $VR_4$, the case of the moves $R_3$ and $VR_3$ are similar. The diagrams $D_1$ and $D_2$ differ only in the small neighborhood, where $D_1$ is as on the left part of Figure~\ref{r3difference}, and $D_2$ is as on the right part of Figure~\ref{r3difference}. 
\begin{figure}[hbt!]
\noindent\centering{
\includegraphics[width=60mm]{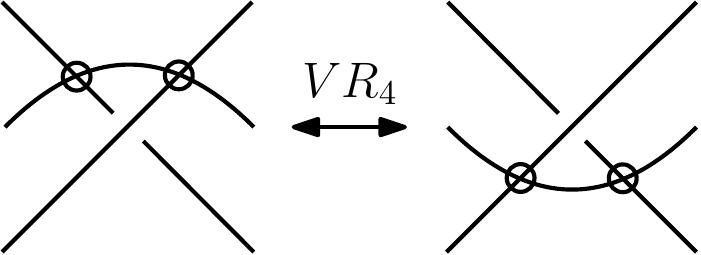}}
\caption{The neighborhood, where $D_1$ and $D_2$ are different: $D_1$ is on the left, $D_2$ is on the right.}
\label{r3difference}
\end{figure}

There are eight possibilities for the orientation of the arcs on Figure~\ref{r3difference}. Turaev (see page~544 in \cite{Tur}) proved that the third Reidemeister move for one orientation of arcs on Figure~\ref{r3difference} can be realized as a sequence of several Reidemeister moves $R_2, VR_2$ and the third Reidemeister move for another fixed orientation of arcs on Figure~\ref{r3difference}. For example, if the top and the bottom arcs on Figure~\ref{r3difference} are oriented from the left to the right, and the middle arcs is oriented from the right to the left, then the Reidemeister move $VR_4$ can be realized as a sequence of moves: 1) two moves $VR_2$, 2) one move $VR_4$, where all arcs are oriented from the left to the right, 3) two moves $VR_2$ (see Figure~\ref{deturim}).
\begin{figure}[hbt!]
\noindent\centering{
\includegraphics[width=140mm]{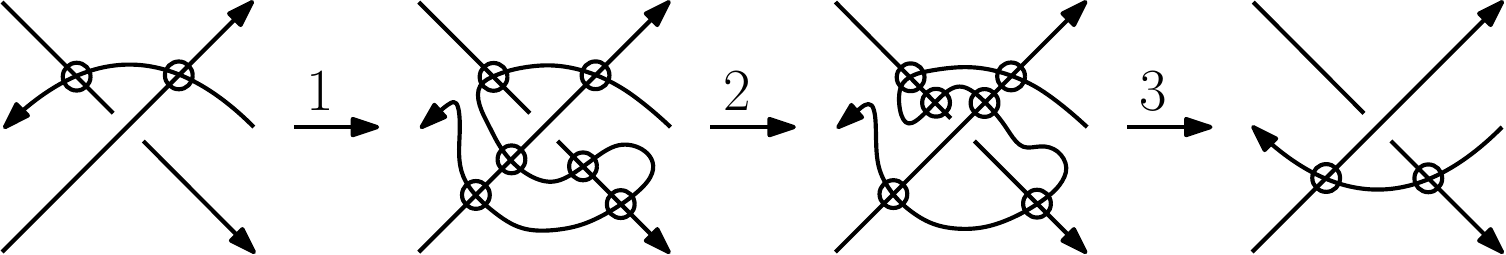}}
\caption{Other orientation of arcs.}
\label{deturim}
\end{figure}
Therefore it is enough to consider only one orientation of the arcs on Figure~\ref{r3difference}. Let us consider the case when all the arcs on Figure~\ref{r3difference} are oriented from the left to the right. 

Suppose that outside of the neighborhood where the $VR_4$-move is applied diagrams $D_1$ and $D_2$ have $n$ arcs. Label the arcs of $D_1$ and $D_2$ (which are not labeled yet) inside of the neighborhood where the $VR_4$-move is applied by additional labels. Then in the neighbourhood where we apply the $VR_4$-move the arcs of the diagram $D_1$ have the labels depicted on Figure~\ref{r31}.
\begin{figure}[hbt!]
\noindent\centering{
\includegraphics[width=55mm]{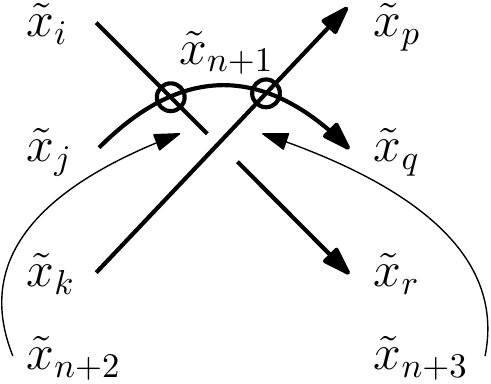}}
\caption{Labels of the arcs in $D_1$.}
\label{r31}
\end{figure}
Therefore $X_{S,V}(D_1)$ is the quotient of $X^{(n+3)}$ ($n$ arcs outside of Figure~\ref{r31} and $3$ new arcs depicted on Figure~\ref{r31}) by the relations which can be written from the part of the diagram outside of Figure~\ref{r31} and $6(m+1)$ relations
\begin{align}
\label{yb11} V(\tilde{x}_j,\tilde{x}_i)&=(\tilde{x}_{n+2},\tilde{x}_{n+1}),\\
\label{yb12} S^{-1}(\tilde{x}_k,\tilde{x}_{n+2})&=(\tilde{x}_{r},\tilde{x}_{n+3}),\\
\label{yb13} V(\tilde{x}_{n+3},\tilde{x}_{n+1})&=(\tilde{x}_{q},\tilde{x}_{p})
\end{align}
from the part of the diagram on Figure~\ref{r31}. Equalities (\ref{yb11}), (\ref{yb12}), (\ref{yb13}) can be rewritten in the following form
\begin{align}
\label{yb21} (id\times V)(\tilde{x}_k,\tilde{x}_j,\tilde{x}_i)&=(\tilde{x}_k,\tilde{x}_{n+2},\tilde{x}_{n+1}),\\
\label{yb22} (S^{-1}\times id)(\tilde{x}_k,\tilde{x}_{n+2},\tilde{x}_{n+1})&=(\tilde{x}_{r},\tilde{x}_{n+3},\tilde{x}_{n+1}),\\
\label{yb23} (id\times V)(\tilde{x}_r,\tilde{x}_{n+3},\tilde{x}_{n+1})&=(\tilde{x}_r,\tilde{x}_{q},\tilde{x}_{p}).
\end{align}
Since for $t=0,1,\dots,m$ the images of the maps $S_t^l,S_t^r$  are words over its arguments in terms of operations of $X$, from equalities~(\ref{yb21}), (\ref{yb22}) one can express the generators from tuples $\tilde{x}_{n+1}, \tilde{x}_{n+2}, \tilde{x}_{n+3}$ as words over generators from tuples $\tilde{x}_{i}, \tilde{x}_{j}, \tilde{x}_{k}$.  So, we can delete the elements from tupels $\tilde{x}_{n+1}, \tilde{x}_{n+2}, \tilde{x}_{n+3}$ from the generating set of $X_{S,V}(D_1)$. Therefore $X_{S,V}(D_1)$ is the quotient of $X^{(n)}$ (we changed $n+3$ by $n$ since we deleted the elements from tupels $\tilde{x}_{n+1}, \tilde{x}_{n+2}, \tilde{x}_{n+3}$ from the generating set) by the relations which can be written from the part of the diagram outside of Figure~\ref{r31} and relations obtained from equalities (\ref{yb21}), (\ref{yb22}), (\ref{yb23}) excluding $\tilde{x}_{n+1}, \tilde{x}_{n+2}, \tilde{x}_{n+3}$. It is easy to see that these equalities will have the form
 \begin{align}
\label{d2in3rdcase}(id\times V)(S^{-1}\times id)(id\times V)(\tilde{x}_k,\tilde{x}_j,\tilde{x}_i)=(\tilde{x}_r,\tilde{x}_q,\tilde{x}_p).
 \end{align}

The algebraic system $X_{S,V}(D_2)$ can be found similarly. Let in the neighbourhood where we apply the $VR_4$-move the arcs of the diagram $D_2$ have the labels depicted on Figure~\ref{r32}.
\begin{figure}[hbt!]
\noindent\centering{
\includegraphics[width=55mm]{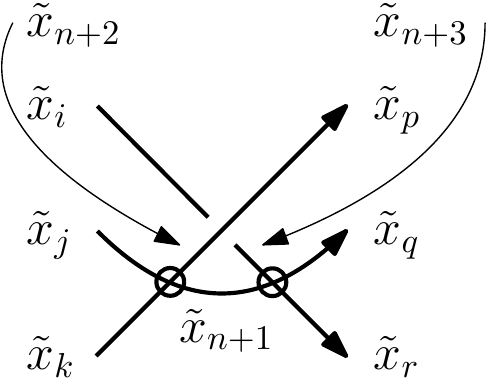}}
\caption{Labels of the arcs in $D_1$.}
\label{r32}
\end{figure}
Then similarly to $X_{S,V}(D_1)$, the algebraic system $X_{S,V}(D_2)$ is the quotient of $X^{(n)}$ by the relations
 \begin{align}
(V\times id)(id\times S^{-1})(V\times id)(\tilde{x}_k,\tilde{x}_j,\tilde{x}_i)=(\tilde{x}_r,\tilde{x}_q,\tilde{x}_p).
 \end{align}
Comparing this equality with equality (\ref{d2in3rdcase}) due to the fact that $(S,V)$ is a virtual multi-switch on $X$ we conclude that $X_{S,V}(D_1)=X_{S,V}(D_2)$.
\end{proof}

Theorem~\ref{ginvariant} gives a powerful tool for constructing invariants for classical and virtual links. In particular, a lot of known invariants can be constructed using Theorem~\ref{ginvariant}. For example, if $X=X_0=F_{\infty}$ is a free group with the free generators $x_1^0,x_2^0,\dots$, and $S,V:X^2\to X^2$ are the maps given by 
\begin{align}
\notag S(x,y)=(y,yxy^{-1}),&&V(x,y)=(y,x),
\end{align}
then $(S,V)$ is a biquandle $1$-switch on $X$. If $L$ is a virtual link presented by a diagram $D$, then the algebraic system $X_{S,V}(D)$ defined in Theorem~\ref{ginvariant}  is the group of a virtual link $L$ introduced by Kauffman in \cite{Kau}. If $D$ represents a classical link, then this group is the classical knot group.

If $X=X_0=FQ_{\infty}$ is the free quandle with the free generators $x_1^0,x_2^0,\dots$, and $S,V:X^2\to X^2$ are the maps given by 
\begin{align}
\notag S(x,y)=(y,x*y),&&V(x,y)=(y,x),
\end{align}
then $(S,V)$ is a biquandle $1$-switch on $X$, and the algebraic system $X_{S,V}(D)$ introduced in Theorem~\ref{ginvariant} is the quandle of a virtual link $L$ (represented by the diagram $D$) introduced by Kauffman in \cite{Kau}. If $D$ represents a classical link $L$, then this quandle is the fundamental quandle of $L$ introduced by Joyce \cite{Joy} and Matveev \cite{Mat}. 

Let $X=F_{\infty}*\mathbb{Z}^{\infty}$ be the free product of the infinitely generated free group $F_{\infty}$ with the free generators $x^0_1,x^0_2,\dots$ and the free abelian group $\mathbb{Z}^{\infty}$ with the canonical generators which are grouped into three groups $x^1_1,x^1_2,\dots$, $x^2_1,x^2_2,\dots$, $x^3_1,x^3_2,\dots$, where $x^3_1=x^3_2=\dots$ Let $X_0=\langle x^0_1,x^0_2,\dots\rangle=F_{\infty}$, $X_1=\langle x^1_1,x^1_2,\dots\rangle=\mathbb{Z}^{\infty}$, $X_2=\langle x^2_1,x^2_2,\dots\rangle=\mathbb{Z}^{\infty}$, $X_3=\langle x^3_1,x^3_2,\dots\rangle=\mathbb{Z}$ and the maps 
$$S,V:X^2\times X_1^2\times X_2^2\times X_3^2\to X^2\times X_1^2\times X_2^2\times X_3^2$$
are given by the formulas
\begin{align}
\notag S(a,b;x,y;p,q;r,s)&=(a b^x a^{-ry}, a^{s};y,x;q,p;s,r),\\ \notag V(a,b;x,y;p,q;r,s)&=(b^{p^{-1}}, a^q;y,x;q,p;s,r),
\end{align}
for $a,b\in X$, $x,y\in X_1$, $p,q\in X_2$, $r,s\in X_3$. Then $(S,V)$ is a virtual $4$-switch on $X$, and the algebraic system $X_{S,V}(D)$ introduced in Theorem~\ref{ginvariant} is the group $G_M(D)$ introduced in \cite{BarMikNes2}.

Let $R=\mathbb{Z}[t^{\pm1},s^{\pm1}]$ be the ring of Laurent polynomials in two variables $t$, $s$. If $X=X^0$ is the free $R$-module with the basis $x^0_1,x^0_2,\dots$, and $S,V:X^2\to X^2$ are the maps given by 
\begin{align}
\notag S(x,y)=(sy,tx+(1-st)y),&&V(x,y)=(y,x)
\end{align}
for $x,y\in X$, then $(S,V)$ is a biquandle $1$-switch on $X$, and the algebraic system $X_{S,V}(D)$ introduced in Theorem~\ref{ginvariant} is the Alexander module of a virtual link represented by $D$. The elementary ideals of this
module determine both the generalized Alexander
polynomial (also known as the Sawollek polynomial) for virtual knots and the classical Alexander polynomial for classical knots.

The examples above show that a lot of known invariants for classical and virtual links can be constructed using Theorem~\ref{ginvariant}. In Section~\ref{newquandlenorm} using Theorem~\ref{ginvariant} we will construct a new quandle invariant for virtual links.

In the examples above we see that a lot of known invariants which are algebraic systems can be constructed as $X_{S,T}$, where $T$ is the twist. This fact reflects the geometric interpretation of the virtual crossing in the virtual link diagram: virtual crossing in the virtual link diagram is a defect of the representation of the virtual knot on the plane. Thus, the following question is natural.
\begin{question}
Let $(S,V)$ be a virtual biquandle $1$-switch on an algebraic system $X$. Is there a biquandle switch $S^{\prime}$ on $X$, such that the invariants $X_{S,V}$ and $X_{S^{\prime}, T}$ are equivalent, where $T$ is the twist on $X$?
\end{question} 
At the moment we do not knot the answer to this question. 

\section{Another representation of $X_{S,V}(D)$}\label{representationshelp}
Multi-switches were introduced in \cite{BarNas2} as a tool for constructing representations of virtual braid groups. Namely, if $(S,V)$ is a virtual multi-switch on an algebraic system $X$, then under additional conditions one can construct a representation 
$$\varphi_{S,V}:VB_n\to {\rm Aut}(X)$$ 
of the virtual braid group $VB_n$ by automorphisms of $X$. In this section we recall the construction of the representation $\varphi_{S,V}$ and show how this representation can be used for finding the set of generators and the set of defining relations of the algebraic system $X_{S,V}(D)$ introduced in Theorem~\ref{ginvariant}.

Let us recall that the virtual braid group $VB_n$ on $n$ strands is the group with $2(n-1)$ generators $\sigma_1,\sigma_2,\dots,\sigma_{n-1}$, $\rho_1,\rho_2,\dots,\rho_{n-1}$ (see Figure~\ref{interpretation})
\begin{figure}[hbt!]
\noindent\centering{
\includegraphics[width=150mm]{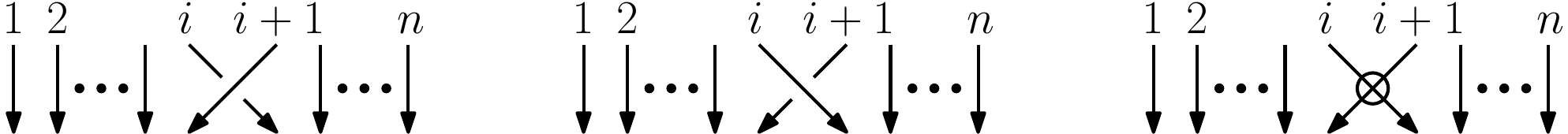}}
\caption{Geometric interpretation of $\sigma_i$ (on the left), $\sigma_i^{-1}$ (in the middle), and $\rho_i$ (on the right).}
\label{interpretation}
\end{figure}
and the following defining relations
\begin{align}
\notag\sigma_i\sigma_{i+1}\sigma_i&=\sigma_{i+1}\sigma_i\sigma_{i+1}&i=1,2,\dots,n-2,\\
\notag\sigma_i\sigma_j&=\sigma_{j}\sigma_i&|i-j|\geq2,\\
\notag\rho_i\rho_{i+1}\rho_i&=\rho_{i+1}\rho_i\rho_{i+1}&i=1,2,\dots,n-2,\\
\notag \rho_i\rho_j&=\rho_{j}\rho_i&|i-j|\geq2,\\
\notag\rho_i^2&=1&i=1,2,\dots,n-1,\\
\notag\rho_{i+1}\sigma_i\rho_{i+1}&=\rho_i\sigma_{i+1}\rho_i&i=1,2,\dots,n-2,\\
\notag\sigma_i\rho_j&=\rho_{j}\sigma_i&|i-j|\geq2.
\end{align}

Let $X$ be an algebraic system, and $X_0,X_1,\dots, X_m$ be subsystems of $X$ such that
\begin{enumerate}
\item   for $i=0,1,\dots,m$ the subsystem $X_i$ is generated by elements $x^{i}_{1},x^{i}_{2},\dots, x^{i}_n$,
\item $\{x^{i}_{1},x^{i}_{2},\dots,x^i_n\}\cap\{x^{j}_{1},x^j_2,\dots,x^j_n\}=\varnothing$ for $i\neq j$,
\item the set of elements $\{x^{i}_j~|~i=0,1,\dots,m, j=1,2,\dots,n\}$ generates $X$.
\end{enumerate}
Let $S=(S_0,S_1,\dots,S_m)$, $V=(V_0,V_1,\dots,V_m)$ be a virtual $(m+1)$-switch on $X$ such that
\begin{align}
\notag S_0=(S_0^l,S_0^r), V_0=(V_0^l,V_0^r)&:X^2 \times X_1^2 \times X_2^2 \times \dots \times X_m^2 \to X^2,\\
\notag S_i=(S_i^l,S_i^r), V_i=(V_i^l,V_i^r)&:X_i^2  \to X_i^2,~\text{for}~i = 1, 2, \dots, m,
\end{align}
and for $i=0,1,\dots,m$ the images of the maps $S_i^l,S_i^r, V_i^l,V_i^r$  are words over its arguments in terms of the operations of $X$. For $j=1,2,\dots,n-1$ denote by $R_j, G_j$ the following maps from $\{x^{i}_j~|~i=0,1,\dots,m, j=1,2,\dots,n\}$ to $X$
\begin{align}
\label{fj} R_j :&\begin{cases}
x^0_{j} \mapsto S_0^l(x^0_{j}, x^0_{j+1}, x^1_{j}, x^1_{j+1}, \dots, x^m_{j}, x^m_{j+1}),\\
x^0_{j+1} \mapsto S_0^r(x^0_{j}, x^0_{j+1}, x^1_{j}, x^1_{j+1}, \dots, x^m_{j}, x^m_{j+1}),\\
x^1_{j} \mapsto S_1^l(x^1_{j}, x^1_{j+1}),\\
x^1_{j+1} \mapsto S_1^r(x^1_{j}, x^1_{j+1}), \\
~~\vdots\\
x^m_{j} \mapsto S_m^l(x^m_{j}, x^m_{j+1}),\\
x^m_{j+1} \mapsto S_m^r(x^m_{j}, x^m_{j+1}),
\end{cases}\end{align}
\begin{align}
\label{gj} G_j :&\begin{cases}
x^0_{j} \mapsto V_0^l(x^0_{j}, x^0_{j+1}, x^1_{j}, x^1_{j+1}, \dots, x^m_{j}, x^m_{j+1}),\\
x^0_{j+1} \mapsto V_0^r(x^0_{j}, x^0_{j+1}, x^1_{j}, x^1_{j+1}, \dots, x^m_{j}, x^m_{j+1}),\\
x^1_{j} \mapsto V_1^l(x^1_{j}, x^1_{j+1}),\\
x^1_{j+1} \mapsto V_1^r(x^1_{j}, x^1_{j+1}), \\
~~\vdots\\
x^m_{j} \mapsto V_m^l(x^m_{j}, x^m_{j+1}),\\
x^m_{j+1} \mapsto V_m^r(x^m_{j}, x^m_{j+1}),
\end{cases}
\end{align}
where all generators which are not explicitly mentioned in $R_j, G_j$ are fixed, i.~e. $R_j(x_k^i)=G_j(x_k^i)=x_k^i$ for $k\neq j,j+1$, $i=0,1,\dots,m$, and assume that $R_j, G_j$ are well defined: since the elements $x^{i}_{1},x^{i}_{2},\dots,x^{i}_{n}$ are not necessary all different, some of these elements can coincide. The fact that $R_j, G_j$ are well defined means that the images of equal elements are equal. For example, if $x_j^i=x_{j+1}^i$, then we assume that
$$S_i^l(x_j^i,x_{j+1}^i)=R_j(x_j^i)=R_j(x_{j+1}^i)=S_i^r(x_j^i,x_{j+1}^i).$$
If for $j=1,2,\dots,n-1$ the maps $R_j$, $G_j$ induce automorphisms of $X$, then we say that $(S,V)$ is \textit{an automorphic virtual multi-switch (shortly, AVMS)} on $X$ with respect to the set of generators $\{x^i_{j}~|~i=0,1,\dots,m,j=1,2,\dots,n\}$. The following theorem is proved in \cite[Theorem~1]{BarNas2}.
\begin{theorem}\label{vautrepr}
Let $(S, V)$ be an AVMS on $X$
with respect to the set of generators $\{x^i_{j}~|~i=0,1,\dots,m,j=1,2,\dots,n\}$. Then the map
$$
\varphi_{S,V} : VB_n \to {\rm Aut}(X)
$$
which is defined on the generators of $VB_n$ as
\begin{align}
\notag\varphi_{S,V}(\sigma_j) = R_j,&&\varphi_{S,V}(\rho_j) = G_j,&&{\text for}~j = 1, 2, \dots, n-1,
\end{align}
where $R_j$, $G_j$ are defined by equalities (\ref{fj}), (\ref{gj}),
is a representation of $VB_n$.
\end{theorem}

The map $\varphi_{S,V}$ from Theorem~\ref{vautrepr} is given explicitly on the generators $\sigma_1,\sigma_2,\dots,\sigma_{n-1}$, $\rho_1,\rho_2,\dots,\rho_{n-1}$ of the virtual braid group $VB_n$ by formulas (\ref{fj}), (\ref{gj}). Let $\beta$ be an arbitrary braid from $VB_n$, and $x$ be an arbitrary element from $X$. In order to find the value $\varphi_{S,V}(\beta)(x)$ express the braid $\beta$ in terms of the generators $\sigma_1,\sigma_2,\dots,\sigma_{n-1}$, $\rho_1,\rho_2,\dots,\rho_{n-1}$, i.~e. write $\beta$ in the form $\beta=\beta_1\beta_2\dots\beta_{k}$, where
$$\beta_i\in\{\sigma_{1},\sigma_2\dots,\sigma_{n-1},\sigma_1^{-1},\sigma_2^{-1},\dots,\sigma_{n-1}^{-1},\rho_1,\rho_2,\dots,\rho_{n-1}\}$$
for $i=1,2,\dots,k$, and then reading $\beta$ from the left to the right calculte the images
$$\begin{CD}
x @>\varphi_{S,V}(\beta_1)>> x_1 @>\varphi_{S,V}(\beta_2)>> x_2@>\varphi_{S,V}(\beta_3)>> \dots@>\varphi_{S,V}(\beta_k)>> x_k.
\end{CD}
$$
The last calculated value $x_k$ is the image of $\varphi_{S,V}(\beta)(x)$. This agreement means that
$$\varphi_{S,V}(\beta_1\beta_2\dots\beta_k)(x)=\varphi_{S,V}(\beta_k)\varphi_{S,V}(\beta_{k-1})\dots\varphi_{S,V}(\beta_1)(x).$$

 Let $X$ be an algebraic system, and $X_0,X_1,\dots, X_m$ be subsystems of $X$ which satisfy conditions (1)-(4) from the beginning of Section~\ref{secinv}. Let $(S,V)$ be a virtual $(m+1)$-switch on $X$ such that
\begin{align}
\notag S_0=(S_0^l,S_0^r), V_0=(V_0^l,V_0^r)&:X^2 \times X_1^2 \times X_2^2 \times \dots \times X_m^2 \to X^2,\\
\notag S_i=(S_i^l,S_i^r), V_i=(V_i^l,V_i^r)&:X_i^2  \to X_i^2,~\text{for}~i = 1, 2, \dots, m,
\end{align}
 and for $i=0,1,\dots,m$ the images of maps $S_i^l,S_i^r, V_i^l, V_i^r$  are words over its arguments in terms of operations of $X$. From this fact follows, in particular that $(S,V)$ induces a virtual $(m+1)$-switch $(S^{(n)}, V^{(n)})$ on $X^{(n)}$ for all $n$. Suppose that for all $n=2,3,\dots$ the virtual $(m+1)$-switch $(S^{(n)}, V^{(n)})$ is AVMS on $X^{(n)}$ with respect to the set of generators $\{x^i_{j}~|~i=0,1,\dots,m,j=1,2,\dots,n\}$. According to Theorem~\ref{vautrepr} for $n=2,3,\dots$ we have representations
$$\varphi_{S^{(n)}, V^{(n)}}:VB_n\to {\rm Aut}\left(X^{(n)}\right).$$
Since $\varphi_{S^{(n)}, V^{(n)}}$ and $\varphi_{S^{(n+1)}, V^{(n+1)}}$ are obtained from the same virtual multi-switch $(S,V)$ restricted to different sets, we have
$$\varphi_{S^{(n+1)}, V^{(n+1)}}|_{VB_n}=\varphi_{S^{(n)}, V^{(n)}},$$
so, automorphisms $\varphi_{S^{(n)}, V^{(n)}}$ ``agree'' with each other for different $n$. Denote by $VB_{\infty}=\bigcup_n VB_n$, and by $\varphi_{S,V}:VB_{\infty}\to {\rm Aut}(X)$ the homomorphism which is equal to $\varphi_{S^{(n)}, V^{(n)}}$ on $VB_n$ (this homomorphism is well defined since $\varphi_{S^{(n)}, V^{(n)}}$ agree with each other). Now we can write $\varphi_{S,V}:VB_n\to {\rm Aut}(X^{(n)})$ meaning the restriction of $\varphi_{S,V}$ to $VB_n$.

If $(S,V)$ is a virtual $(m+1)$-switch on $X$ such that, $S,V$ are biquandle switches on $X\times X_1\times X_2\times\dots\times X_m$, then by Theorem~\ref{ginvariant} one can construct the invariant $X_{S,V}(D)$ for virtual links. The following theorem gives a method how to calculate $X_{S,V}(D)$  using the representation $\varphi_{S,V}:VB_{\infty}\to {\rm Aut}(X)$.
\begin{theorem}\label{usingrepr} Let $\beta \in VB_n$ be a virtual braid, and $D=\hat{\beta}$ be the closure of $\beta$. Then $X_{S,V}(D)$ is equal to the quotient of $X^{(n)}$ by the
relations
$$
\varphi_{S,V}(\beta)(x^i_{j}) = x^i_{j}
$$
for $i = 0,1, \dots, m$, $j = 1,2, \dots, n$.
\end{theorem}
\begin{proof} The braid $\beta$ can be written in the form $\beta=\beta_1\beta_2\dots \beta_k$ where,
$$\beta_i\in\{\sigma_{1},\sigma_2\dots,\sigma_{n-1},\sigma_1^{-1},\sigma_2^{-1},\dots,\sigma_{n-1}^{-1},\rho_1,\rho_2,\dots,\rho_{n-1}\}$$
for $i=1,2,\dots,k$. Label by 
\begin{align}
\notag \tilde{x}_{1}=(x_{1}^0,x_1^1,\dots,x_{1}^m),&& \tilde{x}_{2}=(x_{2}^0,x_2^1,\dots,x_{2}^m),&& \dots&& \tilde{x}_{n}=(x_{n}^0,x_n^1,\dots,x_{n}^m)
\end{align}
 the arcs of the diagram $\hat{\beta}$ which contain the upper points of $\beta_1$ (see Figure \ref{labels1} for the case $\beta_1=\sigma_j$).
 \begin{figure}[hbt!]
\noindent\centering{
\includegraphics[width=75mm]{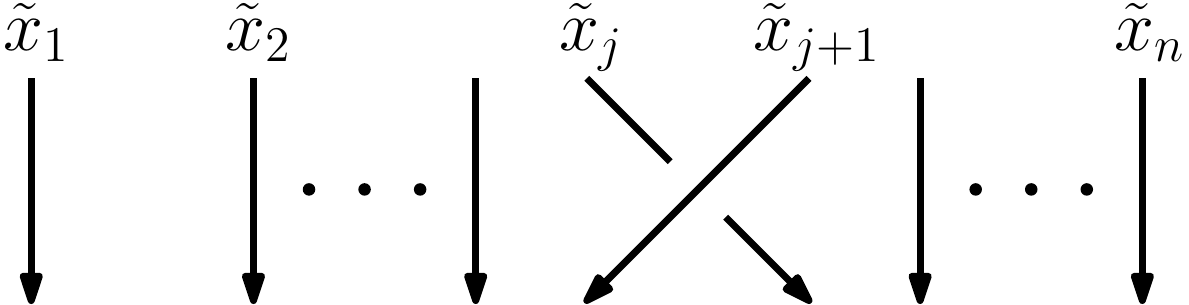}}
\caption{Labels of arcs in $\beta_1=\sigma_{j}$.}
\label{labels1}
\end{figure}\\
In a similar way for $i=2,3,\dots,k$ label by
\begin{align}
\notag \tilde{x}_{(i-1)n+1}&=\left(x_{(i-1)i+1}^0,x_{(i-1)i+1}^1,\dots,x_{(i-1)n+1}^m\right)\\
\notag \tilde{x}_{(i-1)n+2}&=\left(x_{(i-1)i+2}^0,x_{(i-1)i+2}^1,\dots,x_{(i-1)n+2}^m\right)\\
\notag &\dots\\
\notag \tilde{x}_{in}&=\left(x_{in}^0,x_{in}^1,\dots,x_{in}^m\right)
\end{align}
the arcs of the diagram $\hat{\beta}$ which contain the upper points of $\beta_i$. Since the top points of $\beta_{i+1}$ coincide with the bottom points of $\beta_i$, some arcs are labeled several times (see Figure~\ref{labels} for the case $\beta_i=\sigma_j$, where~$\tilde{x}_{(i-1)n+1}=\tilde{x}_{in+1}$).
\begin{figure}[hbt!]
\noindent\centering{
\includegraphics[width=100mm]{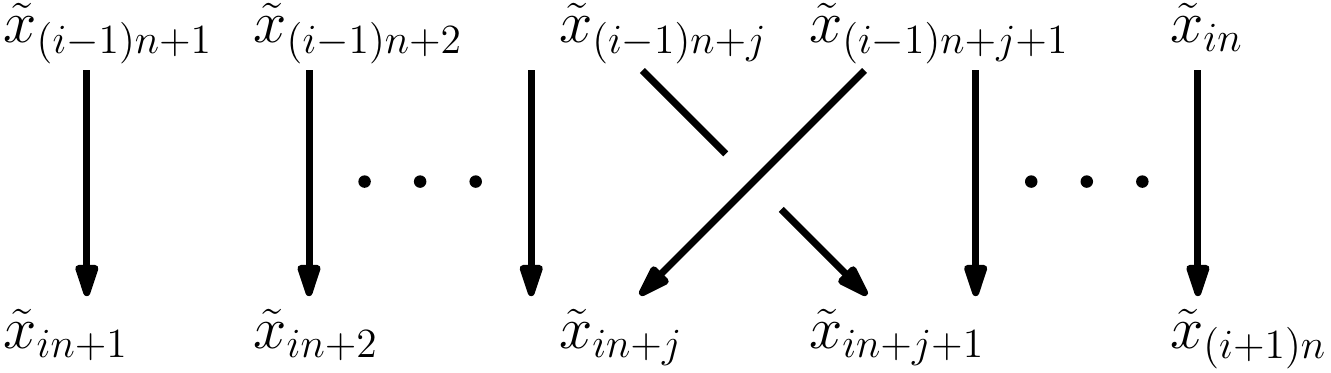}}
\caption{Labels of arcs in $\beta_i=\sigma_{j}$.}
\label{labels}
\end{figure}

Due to Remark~\ref{remark} from Section~\ref{secinv} the algebraic system $X_{S,V}(\hat{\beta})$ is isomorphic to the quotient of $X^{(nk)}$ (for $i=1,2,\dots,k$ each $\beta_i$ gives $n$ tuples of generators) by the $k$ families of relations which can be written from the crossing in $\beta_1,\beta_2,\dots,\beta_k$. The relations which are obtained from $\beta_1$ depending on $\beta_1$ have the following form:\\
if $\beta_1=\sigma_j$, then the relations are
\begin{align}
\notag \tilde{x}_{n+r}&=\tilde{x}_{r}&&\text{for}~r\neq j,j+1,\\
\label{perv} (\tilde{x}_{n+j},\tilde{x}_{n+j+1})&=S(\tilde{x}_{j},\tilde{x}_{j+1}),
\end{align}
\noindent if $\beta_1=\sigma_j^{-1}$, then the relations are 
\begin{align}
\notag \tilde{x}_{n+r}&=\tilde{x}_{r}&&\text{for}~r\neq j,j+1,\\
\label{vtor} (\tilde{x}_{n+j},\tilde{x}_{n+j+1})&=S^{-1}(\tilde{x}_{j},\tilde{x}_{j+1}),
\end{align}
\noindent if $\beta_1=\rho_j$, then the relations are
\begin{align}
\notag \tilde{x}_{n+r}&=\tilde{x}_{r}&&\text{for}~r\neq j,j+1,\\
\label{tret} (\tilde{x}_{n+j},\tilde{x}_{n+j+1})&=V(\tilde{x}_{j},\tilde{x}_{j+1}),
\end{align}
(see Figure~\ref{labels} for $i=1$ and the definition of $X_{S,V}(D)$).  From formulas (\ref{fj}), (\ref{gj}) and Theorem~\ref{vautrepr} (where $\varphi_{S,V}$ is defined) we see that relations (\ref{perv}), (\ref{vtor}), (\ref{tret}) can be rewritten in the unique way
\begin{align}
\label{1stbraid}\tilde{x}_{n+r}=\varphi_{S,V}(\beta_1)(\tilde{x}_r)&& r=1,2,\dots,n,
\end{align}
where $\varphi_{S,V}(\beta_1)(\tilde{x}_r)=(\varphi_{S,V}(\beta_1)(x^0_r),\varphi_{S,V}(\beta_1)(x^1_r),\dots,\varphi_{S,V}(\beta_1)(x^m_r))$.  In a similar way the relations which are obtained from $\beta_2$ depending on $\beta_2$ have the following form:\\
if $\beta_2=\sigma_j$, then the relations are
\begin{align}
\notag \tilde{x}_{2n+r}&=\tilde{x}_{n+r}&&\text{for}~r\neq j,j+1,\\
\label{2perv} (\tilde{x}_{2n+j},\tilde{x}_{2n+j+1})&=S(\tilde{x}_{n+j},\tilde{x}_{n+j+1}),
\end{align}
\noindent if $\beta_2=\sigma_j^{-1}$, then the relations are
\begin{align}
\notag\tilde{x}_{2n+r}&=\tilde{x}_{n+r}&&\text{for}~r\neq j,j+1,\\
\label{2vtor} (\tilde{x}_{2n+j},\tilde{x}_{2n+j+1})&=S^{-1}(\tilde{x}_{n+j},\tilde{x}_{n+j+1}),
\end{align} 
\noindent if $\beta_2=\rho_j$, then the relations are
\begin{align}
\notag\tilde{x}_{2n+r}&=\tilde{x}_{n+r}&&\text{for}~r\neq j,j+1,\\
\label{2tret} (\tilde{x}_{2n+j},\tilde{x}_{2n+j+1})&=V(\tilde{x}_{n+j},\tilde{x}_{n+j+1}).
\end{align}
From formulas (\ref{fj}), (\ref{gj}) and Theorem~\ref{vautrepr}  we see that relations (\ref{2perv}), (\ref{2vtor}), (\ref{2tret}) can be rewritten in the unique way
\begin{align}
\label{2ndbraid}\tilde{x}_{2n+r}=\varphi_{S,V}(\beta_2)(\tilde{x}_{n+r})&& r=1,2,\dots,n.
\end{align}
From relations (\ref{1stbraid}) and (\ref{2ndbraid}) it follows that 
\begin{align}
\notag\tilde{x}_{2n+r}=\varphi_{S,V}(\beta_1\beta_2)(\tilde{x}_{r})&& r=1,2,\dots,n.
\end{align}
Using the same argumentation for $i=1,2,\dots,k-1$ the relations which are obtained from $\beta_i$ can be written in a unique way
\begin{align}
\label{indbraid}\tilde{x}_{in+r}=\varphi_{S,V}(\beta_1\beta_2\dots\beta_i)(\tilde{x}_r)&& r=1,2,\dots,n.
\end{align}
Since in  $\hat{\beta}$ we identify the top points of $\beta$ (which are the top points of $\beta_1$) with the bottom points of $\beta$ (which are the bottom points of $\beta_k$), the relations which are obtained from $\beta_k$ can be written in a unique way
\begin{align}
\label{kndbraid}\tilde{x}_{r}=\varphi_{S,V}(\beta_1\beta_2\dots\beta_k)(\tilde{x}_r)=\varphi_{S,V}(\beta)(\tilde{x}_r)&& r=1,2,\dots,n.
\end{align}
Therefore $X_{S,V}(\hat{\beta})$ is the quotient of $X^{(nk)}$ by the relations (\ref{indbraid}), (\ref{kndbraid}). Using relations (\ref{indbraid}) we can delete the elements from tuples  $\tilde{x}_{in+r}$ for $i=1,2,\dots,{k-1}$ ($n(k-1)$ tuples of generators) from the generating set of $X_{S,V}(D)$. Therefore the algebraic system $X_{S,V}(\hat{\beta})$ is the quotient of $X^{(n)}$ ($nk-n(k-1)=n$) by relations $(\ref{kndbraid})$.
\end{proof}

\section{New quandle invariant for virtual links}\label{newquandlenorm}
In this section using Theorem~\ref{ginvariant} we construct a new quandle invariant for virtual links. Definition of a quandle is given in Section~\ref{multnewmult}.  The information about free products of quandles can be found, for example, in \cite{BarNas}. In \cite[Corollary~2]{BarNas2} we found the following virtual $2$-switch on quandles.
\begin{prop}\label{nweq}Let $Q$ be a quandle, $X_1$ be a trivial quandle, and $X=Q*X_1$ be the free product of $Q$ and $X_1$. Let $S,V:X^2\times X_1^{2}\to X^2\times X_1^{2}$ be the maps defined by
\begin{align}
\notag S(a,b;x,y)=(b,a*b;y,x),&&V(a,b;x,y)=(b*^{-1}x,a*y;y,x)
\end{align}
for $a,b\in X$, $x,y\in X_1$. Then $(S,V)$ is a virtual $2$-switch on $X$. Moreover, the maps $S, V$ are biquandle switches on $X\times X_1$.
\end{prop}
Using the virtual $2$-switch $(S,V)$ introduced in Proposition~\ref{nweq}  we are going to construct a new quandle invariant $\widetilde{Q}(L)$ for virtual links which generalizes the quandle of Manturov \cite{Man2}  and the quandle of Kauffman \cite{Kau}.

Let $X_0=FQ_{\infty}$ be the free quandle on the generators $x_1,x_2,\dots$, $X_1=T_{\infty}$ be the trivial quandle on the elements $y_1,y_2,\dots$, and $X=X_0*X_1$. If for $j=1,2,\dots$ we denote by $x_j^0=x_j$, $x_j^1=y_j$, then it is clear that
\begin{enumerate}
\item   for $i=0,1$ the subsystem $X_i$ is generated by elements $x^{i}_{1},x^{i}_{2},\dots$,
\item $\{x^{0}_{1},x^{0}_{2},\dots\}\cap\{x^{1}_{1},x^1_2,\dots\}=\varnothing$, 
\item the set of elements $\{x^{i}_j~|~i=0,1, j=0,1,\dots\}$ generates $X$,
\item for every permutation $\alpha$ of $\mathbb{N}$ with a finite support the map $x^i_j\mapsto x^i_{\alpha(j)}$ for $i=0,1$, $j=1,2,\dots$ induces an automorphism of $X$. 
\end{enumerate}
Hence, conditions (1)-(4) from Section~\ref{secinv} hold for $X, X_0, X_1$ and for a given virtual link diagram $D$ we can define the algebraic system $X_{S,V}(D)$. Let the virtual link diagram $D$ has $n$ arcs. In order to find $X_{S,V}(D)$ label the arcs of $D$ by the tuples $(x_i,y_i)$ for $i=1,2,\dots,n$. In the neighborhood of some crossings let the labels of arcs are as on Figure~\ref{labelsnewinv}.
\begin{figure}[hbt!]
\noindent\centering{
\includegraphics[width=40mm]{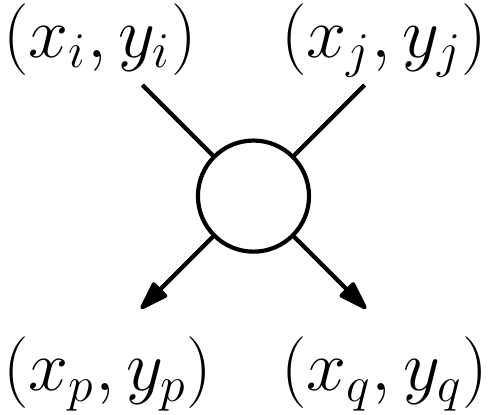}}
\caption{Labels of arcs in $D$ near a crossing.}
\label{labelsnewinv}
\end{figure}\\
Then $X_{S,V}(D)$ is the quotient of $X^{(n)}=FQ_n*T_n$ by the relations which can be written from the crossings of $D$ in the following way (described right before Theorem~\ref{ginvariant}).
\begin{align}
\notag &x_p=x_j, &&x_q=x_i*x_j,&&y_p=y_j, &&y_q=y_i,&&\text{positive crossing},\\
\label{newquandlerelations} &x_p=x_j*^{-1}x_i, &&x_q=x_i,&&y_p=y_j, &&y_q=y_i,&&\text{negative crossing},\\
\notag &x_p=x_j*^{-1}y_i, &&x_q=x_i*y_j,&&y_p=y_j, &&y_q=y_i,&&\text{virtual crossing}.
\end{align}
Since the maps $S,V$ introduced in Proposition~\ref{nweq} are biquandle switches on $X\times X_1$, then from Theorem~\ref{ginvariant} we conclude that $X_{S,V}(D)$ is the virtual link invariant. Let $L$ be a virtual link represented by a virtual link diagram $D$. Denoting by $\widetilde{Q}(L)=X_{S,V}(D)$ we have the following result. 
\begin{theorem}\label{nquaninv2}The quandle $\widetilde{Q}(L)$ is a virtual link invariant.
\end{theorem}
Since $FQ_n*T_n$ is the quotient of the free quandle $FQ_{2n}$ with the generators $x_1,x_2,\dots,x_n,y_1,y_2,\dots,y_n$ by the relations $y_i*y_j=y_i$ for all $i,j=1,2,\dots,n$, the quandle $\widetilde{Q}(L)$ can be written as the quotient of $FQ_{2n}$ (where $n$ is the number of arcs in the diagram $D$ which represents the virtual link $L$) by relations (\ref{newquandlerelations}) which can be written from the crossings of $D$ and relations $y_i*y_j=y_i$ for all $i,j=1,2,\dots,n$.

The quandle $Q(L)$ introduced by Manturov in \cite{Man2} can be found in the following way. Let the diagram $D$ which represents $L$ has $n$ arcs. Label the arcs of $D$ by the symbols $x_1,x_2,\dots,x_n$. Then the quandle $Q(L)$ is the quotient of the free quandle $FQ_{n+1}$ on $n+1$ generators $x_1,x_2,\dots,x_n,y$ by the relations which can be written from the crossings of $D$ in the following way
\begin{align}
\notag &x_p=x_j, &&x_q=x_i*x_j,&&\text{positive crossing},\\
\notag &x_p=x_j*^{-1}x_i, &&x_q=x_i,&&\text{negative crossing},\\
\notag &x_p=x_j*^{-1}y, &&x_q=x_i*y,&&\text{virtual crossing},
\end{align}
where the labels of arcs are as on Figure~\ref{manturovquandle}.
\begin{figure}[hbt!]
\noindent\centering{
\includegraphics[width=20mm]{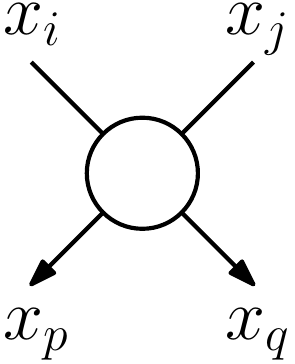}}
\caption{Labels of arcs in $D$ near a crossing.}
\label{manturovquandle}
\end{figure}
Comparing this description with the description of $\widetilde{Q}(L)$ we see, that the map which maps $x_i$ to $x_i$ for $i=1,2,\dots,n$, and maps $y_i$ to $y$ for $i=1,2,\dots,n$ is a surjective homomorphism $\widetilde{Q}(L)\to Q(L)$. It means that the quandle $\widetilde{Q}(L)$ generalizes the quandle $Q(L)$ of Manturov (and therefore it also generalizes the quandle of Kauffman~\cite{Kau}). At the same time $Q(L)$ and $\widetilde{Q}(L)$ are different. For example, if $U$ is a trivial link with $2$ components, then $Q(U)=FQ_3$, while $\widetilde{Q}(U)=FQ_2*T_2$.

Theorem~\ref{usingrepr} says how to use the representation $\varphi_{S,V}$ in order to find $X_{S,V}(D)$. If $(S,V)$ is a virtual $2$-switch described in Proposition~\ref{nweq}, then the representation $\varphi_{S,V}:VB_n\to {\rm Aut}(FQ_n*T_n)$ is denoted by $\varphi_{2Q}$, it is described in \cite[Theorem~3]{BarNas2}, and it has the following form.
\begin{align}\label{quvartin}
\varphi_{2Q}(\sigma_i):\begin{cases}x_i\mapsto x_{i+1},\\
x_{i+1} \mapsto x_i*x_{i+1},\\
y_i \mapsto y_{i+1},\\
y_{i+1}\mapsto y_i,\\
\end{cases}&&
\varphi_{2Q}(\rho_i):\begin{cases}x_i \mapsto x_{i+1}*^{-1}y_i,\\
x_{i+1} \mapsto x_i*y_{i+1},\\
y_i \mapsto y_{i+1},\\
y_{i+1}\mapsto y_i,\\
\end{cases}
\end{align}
From Theorem~\ref{usingrepr} and description (\ref{quvartin}) of the representation $\varphi_{2Q}$ we have the following description of $\widetilde{Q}(L)$.
\begin{theorem}\label{nquaninv}Let $FQ_{n}$ be the free quandle on the set of generators $\{x_1,x_2,\dots,x_n\}$, $T_n=\{y_1,y_2,\dots,y_n\}$ be the trivial quandle, and $\beta\in VB_n$ be a virtual braid. Then the quandle $\widetilde{Q}(\hat{\beta})$ can be written as
\begin{align}
\notag \widetilde{Q}(\hat{\beta})&=\Biggl\langle x_1,x_2,\dots,x_n,y_1,y_2,\dots,y_n~\Biggl|~\begin{array}{rr}
\begin{array}{rrl}
\varphi_{2Q}(\beta)(x_i)&=&x_i,\\
\varphi_{2Q}(\beta)(y_i)&=&y_i,
\end{array}&~~i=1,2,\dots,n,\\
\begin{array}{rrl}y_r*y_s&=&y_r,\end{array}&~~r,s=1,2,\dots,n.
\end{array} \Biggl\rangle,
\end{align}
 where $\varphi_{2Q}$ is defined by (\ref{quvartin}).
\end{theorem}

{\small

\medskip

\medskip

\noindent
Valeriy Bardakov $^{1, 2, 3, 4}$ (bardakov@math.nsc.ru),\\
 Timur Nasybullov $^{1, 2, 3}$ (ntr@math.nsc.ru)\\
 ~\\
$^1$ Tomsk State University, pr. Lenina 36, 634050 Tomsk, Russia,\\
$^2$ Sobolev Institute of Mathematics, Acad. Koptyug avenue 4, 630090 Novosibirsk, Russia,\\
$^3$ Novosibirsk State University, Pirogova 1, 630090 Novosibirsk, Russia,\\
$^4$ Novosibirsk State Agricultural University, Dobrolyubova 160, 630039 Novosibirsk, Russia,\\
}

\end{document}